\documentclass[10pt]{amsart}
\usepackage{amssymb,amsmath,amsthm}
\usepackage{amsrefs}

\makeatletter
\renewcommand{\pod}[1]{\mathchoice
  {\allowbreak \if@display \mkern 8mu\else \mkern 8mu\fi (#1)}
  {\allowbreak \if@display \mkern 8mu\else \mkern 8mu\fi (#1)}
  {\mkern4mu(#1)}
  {\mkern4mu(#1)}
}
\newtheorem{lem}{Lemma}
\newtheorem{prop}{Proposition}
\newtheorem{thm}{Theorem}
\newtheorem*{BZ*}{Theorem (Baker-Zhao)}
\theoremstyle{remark}
\newtheorem*{rem*}{Remark}
\numberwithin{equation}{section}

\title{A note on small gaps between primes in arithmetic progressions}
\author{Deniz A. Kaptan}

\address{Central European University Department of Mathematics and its Applications
1051 Budapest Nador u. 9 Hungary}
\email{denizalikaptan@gmail.com}

\keywords{Distribution of primes, Primes in progressions}
\subjclass[2010]{11N05, 11N13}

\begin{document}
\begin{abstract}
We implement the Maynard-Tao method of detecting primes in tuples to investigate small gaps between primes in arithmetic progressions, with bounds that are uniform over a range of moduli.
\end{abstract}

\maketitle

\section{Introduction}
A long standing problem concerning the distribution of prime numbers is the \emph{prime $k$-tuples conjecture}. We call a set $\mathcal{H}=\{h_1,\ldots,h_k\}$ an \emph{admissible $k$-tuple} if the $h_i$ don't cover all residue classes modulo $p$ for any prime $p$. The prime $k$-tuples conjecture then states that there are infinitely many integers $n$ such that all of the numbers $n+h_i$, $i=1,\ldots,k$ are simultaneously prime. Recently, there have been breakthrough developments towards proving this conjecture. First Zhang \cite{Z2014}, refining a method of Goldston, Pintz and Y\i ld\i r\i m \cite{GPY2009}, proved that for $k$ large enough, the sets $n+\mathcal{H}$ contain two primes infinitely often, thus settling the bounded gaps conjecture. Then Maynard \cite{Ma2015} and Tao (unpublished) independently devised another modification of the Goldston-Pintz-Y\i ld\i r\i m method which could detect $m$ primes in $k$-tuples for any $m$, provided $k$ is large enough.

In this paper we present an implementation of the Maynard-Tao method to yield a corresponding result concerning primes in an arithmetic progression, with a bound that is uniform in the modulus of the progression. The proof goes along the same lines, after tweaking the set-up to pick out only the primes in the arithmetic progression under consideration. The key ingredient will be a Bombieri-Vinogradov type theorem that is tailored to the case at hand, which will be proved in section \ref{BVsect}.

The author would like to thank Roger Baker and Liangyi Zhao for calling his attention to a result of theirs~\cite{BZ2014} which precedes the present work and is of a similar nature. The similarities and differences between the two works will be briefly discussed at the end of the paper.

\section{Notation and setup}\label{notation}
Throughout, the letters $c$ and $C$ will denote constants which need not be the same at every instance. When we need to track constants, we employ subscripts or superscripts.

The method requires that we restrict ourselves to arithmetic progressions in which primes are reasonably well-distributed, i.e.\! progressions to moduli whose associated Dirichlet $L$-functions don't vanish too close to $s=1$.

For the imaginary part $\gamma$ of a zero of an $L$-function, we shall denote $\lvert\gamma\rvert+1$ by $\widetilde{\gamma}$ for the sake of brevity.  We first recall some basic facts concerning zero-free regions of $L$-functions \cite{D2000}*{\S14}. There is a constant $c_0$ (the bounds cited below are known in fact for different constants, but we take $c_0$ to be the minimum of those to simplify notation) such that an $L$-function $L(s,\chi)$ to the modulus $q$ has no zero $\beta+i\gamma$ in the region
\begin{equation}\label{classicalzerofree}
\beta\geq1-\frac{c_0}{\log q\widetilde{\gamma}},
\end{equation}
except possibly a single real zero, which can exist for at most one real character $\chi\pmod{q}$. We call a modulus to which there's such a \emph{primitive} character an \emph{exceptional modulus}, and the corresponding zero an \emph{exceptional zero}. Exceptional moduli are of the form $q=2^\nu p_1\ldots p_m$, where $\nu\leq3$ and $p_1<p_2<\ldots <p_m$ are distinct odd primes, whence, by the Prime Number Theorem, we have $p_m\gg\sum_{p\leq p_m}\log p\gg \log q$.
On the other hand we have, for the real zeros, the unconditional bound
\begin{equation}\label{uncond}
\beta<1-\frac{c_0}{q^{1/2}(\log q)^2}.
\end{equation}
Also, if $\chi_1$ and $\chi_2$ are distinct real primitive characters to moduli $q_1$ and $q_2$ respectively and the corresponding $L$-functions have real zeros $\beta_1$ and $\beta_2$, then the Landau-Page theorem states that these zeros must satisfy
\begin{equation}\label{minzero}
\min(\beta_1,\beta_2)<1-\frac{c_0}{\log q_1q_2}.
\end{equation}

We shall have to confine ourselves to $L$-functions which don't have a zero in the region
\begin{equation}\label{zfreereq}
\beta\geq1-\frac{c^*\log\log X}{\log X},\quad \widetilde{\gamma}\leq\exp\left(c^\sharp\sqrt{\log X}\right)
\end{equation}
for a parameter $X$ and given constants $c^*$ and $c^\sharp$. This is a consequence of \eqref{uncond} when $q\ll\left(\frac{\log X}{(\log\log X)^2}\right)^2$. We suppose that $X$ is large enough in terms of $c^\sharp$ and $c^*$ such that
\begin{equation}\label{Xlarge}
c^\sharp\leq\frac{c_0}{4c^*}\frac{\sqrt{\log X}}{\log\log X}
\end{equation}
holds, and argue that there's at most one modulus $\leq\exp\left(2c^\sharp\sqrt{\log X}\right)$ to which there's a primitive character whose $L$-function vanishes in the region \eqref{zfreereq}. By \eqref{classicalzerofree}, no non-exceptional zeros exist in the region stated, so we only need to consider real zeros. Suppose there are two such moduli $q_1$ and $q_2$, with corresponding real zeros $\beta_1$ and $\beta_2$. Then using \eqref{minzero} we have,
\begin{equation}
\begin{split}
1-\frac{c^*\log\log X}{\log X}&<1-\frac{c_0}{\log q_1q_2}\\
&\leq1-\frac{c_0}{4c^\sharp\sqrt{\log X}},
\end{split}
\end{equation}
which is impossible by \eqref{Xlarge}. We denote this possibly existing unique modulus by $q_0$ and the greatest prime dividing $q_0$ by $p_0$, or set $p_0=1$ in case $q_0$ does not exist. We note that $q_0\gg\left(\frac{\log X}{(\log\log X)^2}\right)^2$, whence $p_0\gg\log\log X$.

Our main parameter $X$ is large enough and $f(X)$ is a given increasing function of $X$ with $f(X)\ll X^{\frac{5}{12}-\frac{5}{6}\vartheta}$ for some positive number $\vartheta<1/2$. The modulus $M$ of the arithmetic progression does not exceed $f(X)$ and is not a multiple of any number in a set $\mathcal{Z}$ of exceptions whose size $Z_f$ satisfies
\begin{equation}
Z_f=\begin{cases}
0,\quad&\text{if $f(X)\ll(\log X)^C$,}\\
1,\quad&\text{if $f(X)\ll\exp\left(c\sqrt{\log X}\right)$,}\\
O\left((\log\log X)^C\right),\quad&\text{otherwise}.
\end{cases}
\end{equation}

We denote characters modulo $q$, $M$, and $qM$ by $\psi$, $\xi$, and $\chi$ respectively. A summation $\sum_\chi^*$ over characters with an asterisk in the superscript denotes that the summation is over primitive characters only.

Put $x=X/M$ and let $W=\prod_{p\leq D_0}p$ be the product of primes not exceeding $D_0=\log\log\log X$, and in turn put $W^\prime=W/(W,P_fM)$ and $V=W^\prime M$, where
\begin{equation}
P_f=\begin{cases}
1,\quad&\text{if $f(X)\ll(\log X)^C$,}\\
p_0,\quad&\text{otherwise}.
\end{cases}
\end{equation}
Also put $R=N^{\frac{1}{2}\vartheta-\delta}$ for some small positive $\delta$. Let $\mathcal{H}=\{h_1,\ldots,h_k\}$ be an admissible $k$-tuple with $\operatorname{diam}(\mathcal{H})<D_0M$ such that $h_i\equiv a\pmod{M}$, $i=1,\ldots,k$ for a given residue class $a\pmod{M}$ coprime to $M$. The weights $\lambda_{d_1,\ldots,d_k}$ are supported on $(\prod_id_i,VP_f)=1$, $\prod_id_i<R$, and $\mu(\prod_id_i)^2=1$ (the last condition implies, of course, that $(d_i,d_j)=1$). We also choose $\nu_0$ such that $(M\nu_0+h_i,W^\prime)=1$ for $i=1,\ldots,k$ (this is possible because $\mathcal{H}$ is admissible).

With these, we will consider the sum
\begin{equation}
S^{(\rho)}=\sum_{\substack{x\leq n<2x\\n\equiv\nu_0\pmod{W^\prime}}}\biggl(\sum_{i=1}^k\chi_\mathbb{P}(nM+h_i)-\rho\biggr)\biggl(\sum_{d_i\mid nM+h_i}\lambda_{d_1,\ldots,d_k}\biggr)^2,
\end{equation}
where $\chi_\mathbb{P}$ is the characteristic function of primes. Clearly, the positivity of $S^{(\rho)}$ implies that for at least one $n\in[x,2x)$, the inner sum is positive, and this establishes the existence of at least $\lfloor\rho+1\rfloor$ primes among the numbers $nM+h_i$, $i=1,\ldots,k$, but $nM$ lies in $[X,2X)$ and each $nM+h_i$ is congruent to $a\pmod{M}$ by the condition on $\mathcal{H}$.

\section{Results}
Our main theorem is the following.
\begin{thm}\label{main}
Let $k$ be a given integer, $\vartheta<1/2$, and $f(X)\ll X^{\frac{5}{12}-\frac{5}{6}\vartheta}$ an increasing function of $X$. Further, let $\mathcal{S}_k$ be the set of all piecewise differentiable functions $\mathbb{R}^k\to\mathbb{R}$ supported on $\mathcal{R}_k=\{(x_1,\ldots,x_k)\in[0,1]^k:\sum_{i=1}^kx_i=1\}$, and put 
\begin{equation}
M_k=\sup_{F\in\mathcal{S}_k}\frac{\sum_{m=1}^kJ^{(m)}_k(F)}{I_k(F)},
\end{equation}
where
\begin{align}
I_k(F)&=\int_0^1\cdots\int_0^1F(t_1,\ldots,t_k)^2dt_1\ldots dt_k,\label{Ik}\\
J_k^{(m)}(F)&=\int_0^1\cdots\int_0^1\Biggl(\int_0^1F(t_1,\ldots,t_k)dt_m\Biggr)^2dt_1\ldots dt_{m-1}dt_{m+1}dt_k.\label{Jkm}
\end{align}
Then, if $X$ is large enough, then for all $M\leq f(X)$, except those which are multiples of numbers in a set of size $Z_f$, all residue classes $a\pmod{M}$ coprime to $M$, and all admissible $k$-tuples $\mathcal{H}=\{h_1,\ldots,h_k\}$ such that $h_i\equiv a\pmod{M}$, $i=1,\ldots,k$, there is a multiple $nM$ of $M$ with $nM\in[X,2X]$ such that at least $r_k=\lceil \vartheta M_k/2\rceil$ of the numbers $nM+h_i$, $i=1,\ldots,k$ are primes.
\end{thm}
We can instantiate this to some concrete cases to deduce certain facts. We denote by $p^\prime_n$ the $n$-th prime that is congruent to $a\pmod{M}$. First note that if $f(X)\leq\exp\left(c\sqrt{\log X}\right)$ for some $c$, we can apply the theorem with $\vartheta$ as close to $1/2$ as we like, and the set of exceptions will be empty or a singleton according as $f(X)\ll (\log X)^C$ for some $C$ or not. In either case taking $k=105$ suffices to produce two primes, by Proposition 4.3 of Maynard \cite{Ma2015}, and likewise any $k$ that produces $r$ primes in Maynard's case does so here too. Since from any admissible tuple $\{h_i\}_i$ we can obtain a tuple $\{Mh_i+a\}_i$ whose members are all congruent to $a\pmod{M}$, with diameter dilated by $M$, we have the following theorems.
\begin{thm}
Let $C$ be a given positive constant. Then if $X$ is sufficiently large, for all $M\ll(\log X)^C$ and all $a$ with $(a,M)=1$, there is a $p^\prime_n\in[X,2X]$ such that
\begin{equation}
p^\prime_{n+1}-p^\prime_n\leq 600M.
\end{equation}
\end{thm}
\begin{thm}
Let $c$ be a given positive constant. Then if $X$ is sufficiently large, for all $M\ll\exp\left(c\sqrt{\log X}\right)$ except those that are a multiple of a single number, and all $a$ with $(a,M)=1$, there is a $p^\prime_n\in[X,2X]$ such that
\begin{equation}
p^\prime_{n+1}-p^\prime_n\leq 600M.
\end{equation}
\end{thm}
\begin{thm}
Let $r$ be a positive integer and $C$ be a given positive constant. Then if $X$ is sufficiently large, for all $M\ll(\log X)^C$ and all $a$ with $(a,M)=1$, there is a $p^\prime_n\in[X,2X]$ such that
\begin{equation}
p^\prime_{n+r}-p^\prime_n\ll r^3e^{4r}M.
\end{equation}
\end{thm}
\begin{thm}
Let $r$ be a positive integer and $c$ be a given positive constant. Then if $X$ is sufficiently large, for all $M\ll\exp\left(c\sqrt{\log X}\right)$ except those that are a multiple of a single number, and all $a$ with $(a,M)=1$, there is a $p^\prime_n\in[X,2X]$ such that
\begin{equation}
p^\prime_{n+r}-p^\prime_n\ll r^3e^{4r}M.
\end{equation}
\end{thm}
When $M$ is allowed to grow as large as a power of $X$ our tuple lengths have to grow and our bounds get much weaker. Suppose $M\ll X^{\frac{5}{12}-\eta}$ for some positive $\eta$. In that case Theorem \ref{main} applies with $\vartheta=6\eta/5$, so that to find $r+1$ primes we need $k$ such that
\begin{equation}\label{mkr}
\frac{3\eta M_k}{5}>r.
\end{equation}
By Proposition 4.3 of Maynard \cite{Ma2015}, we know that
\begin{equation}
M_k>\log k-2\log\log k-2
\end{equation}
when $k$ is sufficiently large. Then we see that if $k\geq C(r/\eta)^2e^\frac{5r}{3\eta}$ for some absolute constant $C$, \eqref{mkr} is satisfied. We take $k=\lceil C(r/\eta)^2e^\frac{5r}{3\eta}\rceil$, take the admissible tuple $\{Mp_{\pi(k)+1}+a,\ldots,Mp_{\pi(k)+k}+a\}$ of diameter $Mk\log k$, and obtain
\begin{thm}\label{gnrlthm}
Let $\eta$ be given with $0<\eta<5/12$, and let $r$ be a positive integer. Then if $X$ is sufficiently large, for all $M\ll X^{\frac{5}{12}-\eta}$ except those that are multiples of numbers in a set of size $\ll (\log X)^C$ , and all $a$ with $(a,M)=1$, there is a $p^\prime_n\in[X,2X]$ such that
\begin{equation}
p^\prime_{n+r}-p^\prime_n\ll \left(\frac{r}{\eta}\right)^3e^\frac{5r}{3\eta}M.
\end{equation}
\end{thm}

In order to prove Theorem \ref{main}, we write
\begin{equation}\label{sumsum}
S^{(\rho)}=S_2-\rho S_1,
\end{equation}
where
\begin{equation}\label{S1}
S_1=\sum_{\substack{x\leq n<2x\\n\equiv \nu_0\pmod{W^\prime}}}\biggl(\sum_{d_i\mid nM+h_i}\lambda_{d_1,\ldots,d_k}\biggr)^2,
\end{equation}
and
\begin{equation}
\begin{split}\label{S2}
S_2&=\sum_{m=1}^kS_2^{(m)}\\
&=\sum_{m=1}^k\sum_{\substack{x\leq n<2x\\n\equiv \nu_0\pmod{W^\prime}}}\chi_\mathbb{P}(nM+h_m)\left(\sum_{d_i\mid nM+h_i}\lambda_{d_1,\ldots,d_k}\right)^2,
\end{split}
\end{equation}
so that we can estimate $S^{(\rho)}$ by the following proposition.
\begin{prop}\label{ASYMPSUMS}
Let $k$ be a given integer and let $X$ be a parameter that is large enough. Let $\lambda_{d_1,\ldots,d_k}$ be defined in terms of a fixed piecewise differentiable function $F$ by
\begin{equation}
\lambda_{d_1,\ldots,d_k}=\biggl(\prod_{i=1}^k\mu(d_i)d_i\biggr)\sum_{\substack{r_1,\ldots,r_k\\d_i\mid r_i\forall i\\(r_i,V)=1\forall i}}\frac{\mu\bigl(\prod_{i=1}^kr_i\bigr)^2}{\prod_{i=1}^k\varphi(r_i)}F\bigl(\frac{\log r_1}{\log R},\ldots,\frac{\log r_k}{\log R}\bigr),
\end{equation}
whenever $(\prod_{i=1}^kd_i,VP_f)=1$, and let $\lambda_{d_1,\ldots,d_k}=0$ otherwise. Moreover, let F be supported on $\mathcal{R}_k=\{(x_1,\ldots,x_k)\in[0,1]^k:\sum_{i=1}^kx_i\leq 1\}$. Then we have
\begin{align}
S_1&=\left(1+o(1)\right)\frac{\varphi(VP_f)^kX(\log R)^k}{V(VP_f)^k}I_k(F),\\
S_2&=\left(1+o(1)\right)\frac{\varphi(VP_f)^kX(\log R)^{k+1}}{V(VP_f)^k\log X}\sum_{m=1}^kJ^{(m)}_k(F),
\end{align}
provided $I_k(F)\neq 0$ and $J_k^{(m)}(F)\neq 0$ for each $m$, where $I_k(F)$ and $J_k^{(m)}(F)$ are given by \textup{(\ref{Ik})} and \textup{(\ref{Jkm})} respectively.
\end{prop}
From this, Theorem \ref{main} immediately follows.
\begin{proof}[Proof of Theorem \ref{main}]
Let $\mathcal{S}_k$  and $M_k$ be as in Theorem \ref{main}. Then for any $\delta>0$, we can find $F_0\in\mathcal{S}_k$ such that $\sum_{m=1}^kJ^{(m)}_k(F_0)>(M_k-\delta)I_k(F_0)$. With this $F_0$, we have, by (\ref{sumsum}) and Proposition \ref{ASYMPSUMS}
\begin{align*}
S^{(\rho)}=&\frac{\varphi(VP_f)^kX(\log R)^k}{V(VP_f)^k}\biggl(\frac{\log R}{\log N}\sum_{m=1}^kJ^{(m)}_k(F_0)-\rho I_k(F_0)+o(1)\biggr)\\
\geq&\frac{\varphi(VP_f)^kX(\log R)^k}{V(VP_f)^k}I_k(F)\biggl(\bigl(\frac{\vartheta}{2}-\delta\bigr)\bigl(M_k-\delta\bigr)-\rho+o(1)\Biggr).
\end{align*}
If $\rho=\vartheta M_k/2-\delta^\prime$, then with $\delta$ sufficiently small, we have $S^{(\rho)}>0$ for all large enough $X$, implying that at least $\lfloor\rho+1\rfloor$ of the $nM+h_i$ are prime. Since $\lfloor\rho+1\rfloor=\lceil \vartheta M_k/2\rceil$ for $\delta^\prime$ small enough, we obtain our result.
\end{proof}

\section{A Bombieri-Vinogradov type theorem}\label{BVsect}
Throughout this section, a summation $\sum_\chi^*$ over characters with an asterisk in the superscript denotes that the sum runs over primitive characters only. We quote here a zero-density result \cite{IK2004}*{Theorem 10.4 and the following remark} which we will need in our proof.
\begin{thm}\label{densitythm}
Let $m$ be given and $N(1-\delta,T,\chi)$ be the number of zeros $\beta+i\gamma$ of $L(s,\chi)$ in the region $1-\delta\leq\beta$, $\lvert\gamma\rvert\leq T$. Put
\begin{equation}
N(1-\delta,m,Q,T)=\sum_{\substack{q\leq Q\\(q,m)=1}}\sideset{}{^*}\sum_{\psi\pmod{q}}\sum_{\xi\pmod{m}}N(1-\delta,T,\psi\xi).
\end{equation}
Then for $\delta<1/2$ and any $\varepsilon>0$, we have
\begin{equation}
N(1-\delta,m,Q,T)\ll\left((mQT)^{2\delta}+(mQ^2T)^{c(\delta)\delta}\right)(\log mQT)^A,
\end{equation}
for some constant $A$, where
\begin{equation}
c(\delta)=\min\left(\frac{3}{1+\delta},\frac{3}{2-3\delta}\right).
\end{equation}
\end{thm}

We estimate the number of the moduli we will have to exclude in the following proposition.
\begin{prop}\label{exception}
Let $c^*$ and $c^\sharp$ be given constants. There is a set $\mathcal{Z}$ of exceptions with $\lvert\mathcal{Z}\rvert\ll(\log X)^C$ such that if $X$ is large enough, then for all $M\leq f(X)$ that is not a multiple of any number in $\mathcal{Z}$ and all $q\leq\exp\left(c^\sharp\sqrt{\log X}\right)$ with $(q,Mp_0)=1$, the $L$-functions $L(s,\psi\xi)$, where $\psi\pmod{q}$ is primitive and $\xi$ is any character $\pmod{M}$, have no zeros in the region $1-\tfrac{c^*\log\log X}{\log X}\leq\beta\leq1$, $\lvert\gamma\rvert\leq\exp\left(c^\sharp\sqrt{\log X}\right)$. The set $\mathcal{Z}$ can have elements $\leq\exp\left(c^\sharp\sqrt{\log X}\right)$ only if $q_0$ exists, in which case those elements are all multiples of $p_0$.
\end{prop}
\begin{proof}
Suppose $M$ is a modulus such that for some character $\xi\pmod{M}$ and a primitive character $\psi\pmod{q}$, $L(s,\psi\xi)$ has a zero in the region indicated. Then $\psi\xi$ must be induced by a character of the form $\psi\xi^*$, where $\xi^*\pmod{m}$ is a primitive character modulo $m\mid M$. We estimate the number of such $m$. Let
\begin{multline}
\mathcal{Z}=\{m\leq f(X):\text{ there exist }q\leq\exp\left(c^\sharp\sqrt{\log X}\right)\text{ and }\chi\pmod{mq}\text{ primitive }\\
\text{ with }{(q,mp_0)=1},\text{ such that }L(\beta+i\gamma,\chi)=0\text{ for some }\beta>1-\tfrac{c^*\log\log X}{\log X}\}
\end{multline}
be the set of exceptions whose size we wish to bound. We divide the ranges $1\leq m\leq f(X)$, $1\leq q\leq\exp\left(c^\sharp\sqrt{\log X}\right)$, and $\widetilde{\gamma}\leq\exp\left(c^\sharp\sqrt{\log X}\right)$ into dyadic segments $[M_\lambda/2,M_\lambda)$ $[Q_\mu/2,Q_\mu)$ and $[T_\nu/2,T_\nu)$ respectively. Then,
\begin{equation}
\#\mathcal{Z}\leq\sum_{\lambda,\mu,\nu}\sum_{m\leq M_\lambda}\sum_{\substack{q\leq Q_\mu\\(q,mp_0)=1}}\:\sideset{}{^*}\sum_{\chi\pmod{qm}}N(1-\tfrac{c^*\log\log X}{\log X},T_\nu,\chi).
\end{equation}
Using Theorem \ref{densitythm} with $m=1$ and $M_\lambda Q_\mu$ in place of $Q$, the above is
\begin{equation}
\ll \sum_{\lambda,\mu,\nu}(M_\lambda^2Q_\mu^2 T_\nu)^\frac{12c^*\log\log X}{5\log X}(\log Q_\mu T_\nu)^C\ll (\log X)^C,
\end{equation}
where $C$ and the implicit constant depend on $c^*$ and $c^\sharp$. Now suppose that $X$ is large enough to satisfy \eqref{Xlarge}. Then if $m\in \mathcal{Z}$ with $m\leq\exp\left(c^\sharp\sqrt{\log X}\right)$, so that $mq\leq\exp\left(2c^\sharp\sqrt{\log X}\right)$, then by the discussion in Section~\ref{notation}, $L(s,\chi)=0$ with primitive $\chi\pmod{mq}$ implies $mq=q_0$, and since $p_0\nmid q$, we have $p_0\mid m$.
\end{proof}
\begin{rem*}
If $f(X)\leq\exp\left(c\sqrt{\log X}\right)$ for some constant $c$, then choosing $c^\sharp$ to be such a constant, one sees that $\mathcal{Z}$ can be taken to be at most a singleton.
\end{rem*}
\begin{rem*}
We see that asymptotically almost all moduli remain after exceptions, because the excluded moduli number at most $\ll\frac{f(X)}{(\log\log X)}$, since $p_0\geq\log\log X$.
\end{rem*}

Using this proposition, we prove the following
\begin{thm}\label{BV1}
Let $A$ be a given positive number. There exists a positive number $B$ such that for all $M\leq f(X)$, except those that are multiples of numbers in a set of size $Z_f$, we have
\begin{equation}
\sum_{\substack{q\leq \frac{X^{1/2}}{M^{6/5}}(\log X)^{-B}\\(q,Mp_0)=1}}\max_{(a,qM)=1}\;\Big\lvert\psi(X;qM,a) - \frac{\psi(X)}{\varphi(qM)}\Big\rvert\ll\frac{X}{\varphi(M)}(\log X)^{-A},
\end{equation}
where the implicit constants depend on $A$.
\end{thm}
\begin{proof}
Let $c^*$ be a constant to be specified later in terms of $A$, and pick $c^\sharp$ arbitrarily (or, in case $f(X)\leq\exp\left(c\sqrt{\log X}\right)$ for some $c$, pick $c^\sharp$ according to the first remark following Proposition \ref{exception}), so that Proposition \ref{exception} furnishes us with a set $\mathcal{Z}$ of size $Z_f$. Then if $M$ is not a multiple of any number in $\mathcal{Z}$, $q\leq\exp\left(c^\sharp\sqrt{\log X}\right)$ and $(q,Mp_0)=1$ then no $L(s,\psi\xi)$ with $\psi$ primitive has a zero in the region $\beta\geq 1-\frac{c^*\log\log X}{\log X}$, $\widetilde{\gamma}\leq\exp\left(c^\sharp\sqrt{\log X}\right)$. We put $\Omega=X^{1/2}M^{-6/5}(\log X)^{-B}$ for the sake of brevity. We have
\begin{equation}
\psi(X;qM,a)=\frac{1}{\varphi(qM)}\sum_{\chi\pmod{qM}}\overline{\chi}(a)\psi(X,\chi)
\end{equation}
and
\begin{equation}
\lvert\psi(X,\chi_0)-\psi(X)\rvert\leq\sum_{\substack{n\leq X\\(n,qM)>1}}\Lambda(n)\ll(\log qM)(\log X),
\end{equation}
so it suffices to consider, within acceptable error,
\begin{equation}\label{BV1target}
\sum_{\substack{q\leq\Omega\\(q,Mp_0)=1}}\;\max_{(a,qM)=1}\;\Big\lvert\frac{1}{\varphi(qM)}\sum_{\substack{\chi\pmod{qM}\\\chi\neq\chi_0}}\overline{\chi}(a)\psi(X,\chi)\Big\rvert.
\end{equation}
Since $(M,q)=1$, we can factorize $\chi$ as $\psi\xi$, where $\psi$ and $\xi$ are characters to the moduli $q$ and $M$ respectively (there is no danger of confusing $\psi(n)$ with $\psi(X;q,a)$, nor with $\psi(X,\chi)$), so that \eqref{BV1target} is
\begin{equation}
\sum_{\substack{q\leq\Omega\\(q,Mp_0)=1}}\;\max_{(a,qM)=1}\;\Big\lvert\frac{1}{\varphi(qM)}\sum_{\substack{\psi\pmod{q}\\\xi\pmod{M}\\\psi\xi\neq\chi_0}}\overline{\psi\xi}(a)\psi(X,\psi\xi)\Big\rvert.
\end{equation}
We replace each character $\psi$ with the primitive character $\psi^*$ inducing it. This leads to an error of
\begin{equation}
\sum_{\substack{q\leq\Omega\\(q,Mp_0)=1}}\frac{1}{\varphi(qM)}\sum_{\substack{\psi\pmod{q}\\\xi\pmod{M}}}\sum_{\substack{n\leq X\\(n,q)>1}}\Lambda(n)\ll \frac{X^{1/2}}{M^{6/5}}\exp\left(-c^\sharp\sqrt{\log X}\right)(\log X)^2,
\end{equation}
and this is acceptable. Using the explicit formula for $\psi(X,\chi)$ in the form
\begin{equation}
\psi(X,\chi)=-\sum_{\substack{\lvert\gamma_\chi\rvert\leq X^{1/2}\\\beta_\chi>{1/2}}}\frac{X^{\rho_\chi}}{\rho_\chi}+O(X^{1/2}(\log X)^2),
\end{equation}
we are left to bound
\begin{equation}
\frac{1}{\varphi(M)}\sum_{\substack{q\leq\Omega\\(q,Mp_0)=1}}\frac{1}{\varphi(q)}\sum_{\xi\pmod{M}}\:\sum_{\psi\pmod{q}}\sum_{\substack{\lvert\gamma_{\psi^*\xi}\rvert\leq X^{1/2}\\\beta_{\psi^*\xi}>{1/2}}}\frac{X^{\beta_{\psi^*\xi}}}{\lvert\rho_{\psi^*\xi}\rvert}.
\end{equation}
We rearrange the sum according to the moduli of the primitive characters $\psi^*$ that occur, hence after relabelling the dummy variables so that $q$ is now the modulus of $\psi^*$, we have
\begin{equation}
\begin{split}
&\frac{X}{\varphi(M)}\:\sum_{\substack{q\leq\Omega\\(q,Mp_0)=1}}\sum_{\xi\pmod{M}}\:\sideset{}{^*}\sum_{\psi\pmod{q}}\sum_{\substack{\lvert\gamma_{\psi\xi}\rvert\leq X^{1/2}\\\beta_{\psi\xi}>{1/2}}}\frac{X^{-(1-\beta_{\psi\xi})}}{\lvert\rho_{\psi\xi}\rvert}\sum_{\substack{k\leq\Omega/q\\(k,Mp_0)=1}}\frac{1}{\varphi(kq)}\\
\ll&\frac{X(\log X)^2}{\varphi(M)}\sum_{\substack{q\leq\Omega\\(q,Mp_0)=1}}\sum_{\xi\pmod{M}}\:\sideset{}{^*}\sum_{\psi\pmod{q}}\sum_{\substack{\lvert\gamma_{\psi\xi}\rvert\leq X^{1/2}\\\beta_{\psi\xi}>{1/2}}}\frac{X^{-(1-\beta_{\psi\xi})}}{q\lvert\rho_{\psi\xi}\rvert}.
\end{split}
\end{equation}
We divide the ranges for $q$ and $\widetilde{\gamma}$ into dyadic segments, and the range for $\beta$ into segments of length $(\log X)^{-1}$ as follows.
\begin{equation}
q\in[Q_\mu/2,Q_\mu), \qquad \widetilde{\gamma}\in[T_\nu/2,T_\nu), \qquad 1-\beta\in[\delta_\lambda-(\log X)^{-1},\delta_\lambda),
\end{equation}
where $2\leq Q_\mu=2^\mu<2\Omega$, $2\leq T_\nu=2^\nu<2X^{1/2}$ and $(\log X)^{-1}\leq\delta_\lambda=\lambda(\log X)^{-1}\leq1/2$. So our expression is
\begin{equation}
\ll\frac{X(\log X)^5}{\varphi(M)}\sup_{(\lambda,\mu,\nu)}\frac{N^*(1-\delta_\lambda,M,Q_\mu,T_\nu)}{Q_\mu T_\nu}X^{-\delta_\lambda},
\end{equation}
where
\begin{equation}
N^*(1-\delta_\lambda,M,Q_\mu,T_\nu)=\sum_{\substack{Q_\mu/2<q\leq Q_\mu\\(q,Mp_0)=1}}\:\sideset{}{^*}\sum_{\psi\pmod{q}}\sum_{\xi\pmod{M}}N(1-\delta_\lambda,T_\nu,\psi\xi).
\end{equation}
Thus we need to show, for all triples $(\lambda,\mu,\nu)$, dropping the subscripts for economy of notation, the upper bound
\begin{equation}
N^*(1-\delta,M,Q,T)\ll QTX^\delta(\log X)^{-A-5}.
\end{equation}
To this end we use the Theorem \ref{densitythm}, which for our ranges of $Q$ and $T$ yields,
\begin{equation}\label{density}
N^*(1-\delta,M,R,T)\ll\left((MQT)^{2\delta}+(MQ^2T)^{c(\delta)\delta}\right)(\log X)^{C^\prime},
\end{equation}
where $C^\prime$ is an absolute constant.

Since for $0\leq\delta\leq1/2$, we have
\begin{equation}
\frac{(MQT)^{2\delta}}{QT}(\log X)^{C^\prime}\ll M^{2\delta}(\log X)^{C^\prime},
\end{equation}
the contribution of the first term on the right hand side of (\ref{density}) is acceptable if $\delta\geq\frac{2}{15}$, say. So we only need to show
\begin{equation}\label{BV1goal}
\frac{(MQ^2T)^{c(\delta)\delta}}{QT}\ll X^\delta(\log X)^{-(A+C^\prime+5)}
\end{equation}
for $0\leq\delta\leq1/2$, and
\begin{equation}\label{BV1goal2}
\frac{(MQT)^{2\delta}}{QT}\ll X^\delta(\log X)^{-(A+C^\prime+5)}
\end{equation}
for $0\leq\delta\leq\frac{2}{15}$.

If $\frac{1}{4}\leq\delta\leq\frac{1}{2}$, we have $c(\delta)=\frac{3}{1+\delta}$. Here $6\delta/(1+\delta)-1\leq2\delta$, $3\delta/(1+\delta)-1\leq0$ and $3\delta/(1+\delta)\leq\frac{12}{5}\delta$, so
\begin{equation}\label{BV1range1}
\begin{split}
\frac{(MQ^2T)^{\frac{3\delta}{1+\delta}}}{QT}&\ll\frac{(MQ^2)^\frac{3\delta}{1+\delta}}{Q}\\
&\ll M^{\frac{12}{5}\delta}\left(\frac{X^{1/2}}{M^\frac{6}{5}(\log X)^B}\right)^{2\delta}\\
&\ll X^\delta(\log X)^{-2\delta B}\\
&\ll X^\delta(\log X)^{-(A+C^\prime+5)},
\end{split}
\end{equation}
if $B\geq 2(A+C^\prime+5)$.

If $\frac{2}{15}\leq\delta\leq\frac{1}{4}$, we have $c(\delta)=3/(2-3\delta)$. Here also $3\delta/(2-3\delta)\leq\frac{12}{5}\delta$, $0\leq6\delta/(2-3\delta)-1\leq\frac{4}{5}\delta$ and $3\delta/(2-3\delta)-1\leq0$, so
\begin{equation}\label{BV1range2}
\begin{split}
\frac{(MQ^2T)^{\frac{3\delta}{2-3\delta}}}{QT}&\ll\frac{(MQ^2)^\frac{3\delta}{2-3\delta}}{Q}\\
&\ll M^{\frac{12}{5}\delta}\left(\frac{X^{1/2}}{M^\frac{6}{5}(\log X)^B}\right)^{\frac{4}{5}\delta}\\
&\ll M^{\frac{36}{25}\delta}X^{\frac{2}{5}\delta}(\log X)^{-\frac{4}{5}\delta B},
\end{split}
\end{equation}
and this is $\ll X^\delta(\log X)^{-(A+C^\prime+5)}$ if $M\ll X^{5/12}$ and $B\geq\frac{75}{8}(A+C^\prime+5)$.

Now suppose $\delta\leq\tfrac{2}{15}$. Then $6\delta/(2-3\delta)-1\leq-\tfrac{1}{2}$ and $3\delta/(2-3\delta)\leq 2\delta$, so
\begin{equation}\label{deltagoal}
\frac{(MQ^2T)^{\frac{3\delta}{2-3\delta}}}{QT}\ll M^{2\delta}(QT)^{-1/2},
\end{equation}
as well as
\begin{equation}
\frac{(MQT)^{2\delta}}{QT}\ll M^{2\delta}(QT)^{-1/2}.
\end{equation}

Now if $M\ll X^{\frac{5}{12}}$ and $QT\geq\exp\left(c^\sharp\sqrt{\log X}\right)$, the right hand side is
\begin{equation}
\leq M^{2\delta}\exp\left(-\tfrac{c^\sharp}{2}\sqrt{\log X}\right)\ll X^\delta(\log X)^{-(A+C^\prime+5)}.
\end{equation}
Otherwise, if $QT\leq\exp\left(c^\sharp\sqrt{\log X}\right)$, we use the fact that $\delta\geq\frac{c^*\log\log X}{\log X}$ by our assumption on $M$, and we have
\begin{equation}
\begin{split}
\leq\left(\frac{M^2}{X}\right)^\delta&\leq\exp\left(-\frac{c^*}{5}\log\log X\right)\\
&\leq(\log X)^{-(A+C^\prime+5)},
\end{split}
\end{equation}
provided $c^*\geq 5(A+C^\prime+5)$.
\end{proof}
\begin{rem*}
Note that when $M$ indeed reaches $X^{5/12}$, the sum is vacuous and the theorem is trivial. We will apply it with $M\ll X^{\frac{5}{12}-\frac{5}{6}\vartheta}$ for some positive $\vartheta$ to get ``level of distribution'' $\vartheta$.
\end{rem*}

For the shorter range $M\leq(\log X)^C$, we can simply use the classical Bombieri-Vinogradov theorem (see, for instance, \cite{D2000}*{\S 28}) with $A+C$ in place of $A$, and gain a factor of $\phi(M)$ without any further modifications.
\begin{thm}\label{BV3}
Let $A$ be a given positive number and let $M\ll(\log X)^C$ be an integer. Then there is a positive number $B$ such that
\begin{equation}
\sum_{\substack{q\leq X^{1/2}(\log X)^{-B}\\(q,M)=1}}\;\max_{(a,qM)=1}\;\Big\lvert\psi(X;qM,a) - \frac{\psi(X)}{\varphi(qM)}\Big\rvert\ll\frac{X}{\varphi(M)}(\log X)^{-A},
\end{equation}
where the implicit constant depends on $A$ and $C$.
\end{thm}

We would like to express these results in a unified fashion. To that end, given an increasing function $f(X)$ of $X$ such that $f(X)\ll X^{\frac{5}{12}-\frac{5}{6}\vartheta}$ with $\vartheta>0$, we introduce the following notation.
\begin{equation}
e_f=\begin{cases}
\frac{1}{2},\quad&\text{if $f(X)\leq\exp\left(C\sqrt{\log X}\right)$,}\\
\vartheta,\quad&\text{otherwise}.
\end{cases}
\end{equation}
With this we have
\begin{thm}\label{UnifiedBV}
Let $A$ be given positive numbers and $f(X)$ an increasing function of $X$ satisfying $f(X)\ll X^C$ with $C<5/12$. Then for all $M\leq f(X)$, except multiples of numbers in a set of size at most $Z_f$, and all $\delta>0$, we have
\begin{equation}
\sum_{\substack{q\leq X^{e_f-\delta}\\(q,MP_f)=1}}\;\max_{(a,qM)=1}\;\Big\lvert\psi(X;qM,a) - \frac{\psi(X)}{\varphi(qM)}\Big\rvert\ll\frac{X}{\varphi(M)}(\log X)^{-A}.
\end{equation}
\end{thm}
Now we are in a position to prove our main proposition.

\section{Proof of Proposition \ref{ASYMPSUMS}}
This section consists of lemmata that establish Proposition \ref{ASYMPSUMS}. They follow the corresponding results in \cite{Ma2015} \emph{mutatis mutandis}. In \cite{Ma2015}, the parameter $W$ features in a dual role: first in that the weights $\lambda_{d_1,\ldots,d_k}$ are supported for $(\prod d_i,W)=1$, and second in the ``$W$-trick'', i.e.\! in restricting $n$ to $n\equiv\nu_0\pmod{W}$. In our case we have $VP_f$ in the first role and $W^\prime$ in the second.
\begin{lem}\label{S11}
Let
\begin{equation}
y_{r_1,\ldots,r_k}=\left(\prod_{i=1}^k\mu(r_i)\varphi(r_i)\right)\sum_{\substack{d_1,\ldots,d_k\\r_i\mid d_i\forall i}}\frac{\lambda_{d_1,\ldots,d_k}}{\prod_{i=1}^kd_i},
\end{equation}
and let $y_{max}=\sup_{r_1,\ldots,r_k}\lvert y_{r_1,\ldots,r_k}\rvert$. Then we have
\begin{equation}
S_1=\frac{X}{V}\sum_{u_1,\ldots,u_k}\frac{y_{u_1,\ldots,u_k}^2}{\prod_{i=1}^k\varphi(u_i)}+O\left(\frac{y_{max}^2\varphi(VP_f)^kX(\log X)^k}{V(VP_f)^kD_0}\right).
\end{equation}
\end{lem}
\begin{proof}
We start by rearranging the sum on the right hand side of (\ref{S1}) to obtain
\begin{equation}
S_1=\sum_{\substack{d_1,\ldots,d_k\\e_1,\ldots,e_k}}\lambda_{d_1,\ldots,d_k}\lambda_{e_1,\ldots,e_k}\sum_{\substack{x\leq n<2x\\n\equiv \nu_0\pmod{W^\prime}\\ [d_i,e_i]\mid nM+h_i}}1.
\end{equation}
Now when $W^\prime,[d_1,e_1],\ldots [d_k,e_k]$ are pairwise coprime, the inner sum is over a single residue class modulo $q=W^\prime\prod_i[d_i,e_i]$ by the Chinese Remainder Theorem, otherwise it is empty, in the case $p\mid(W^\prime,[d_i,e_i])$ because of the condition $(W^\prime,M\nu_0+h_i)=1$, and in the case $p\mid ([d_i,e_i],[d_j,e_j])$ because it would imply $p\mid h_i-h_j$, but $h_i-h_j=fM$ for some $f<D_0$ since $h_i$ and $h_j$ lie in the same residue class modulo $M$, but $p\nmid M$ and $p$ can't be a prime less than $D_0$ by the support of $\lambda$. Since $f<D_0$ by the diameter of $\mathcal{H}$, we deduce that there's no contribution when $([d_i,e_i],[d_j,e_j])>1$.

Thus the inner sum is $x/q+O(1)$, and we have
\begin{equation}
S_1=\frac{X}{V}\sideset{}{^\prime}\sum_{\substack{d_1,\ldots,d_k\\e_1,\ldots,e_k}}\frac{\lambda_{d_1,\ldots,d_k}\lambda_{e_1,\ldots,e_k}}{\prod_{i=1}^k[d_i,e_i]}+O\Bigl(\sideset{}{^\prime}\sum_{\substack{d_1,\ldots,d_k\\e_1,\ldots,e_k}}\lvert\lambda_{d_1,\ldots,d_k}\lambda_{e_1,\ldots,e_k}\rvert\Bigr),
\end{equation}
where $\textstyle\sideset{}{^\prime}\sum$ denotes the coprimality restrictions. The error term is plainly
\begin{equation}\label{err}
\ll\lambda^2_{max}\left(\sum_{d<R}\tau_k(d)\right)^2\ll\lambda^2_{max}R^2(\log X)^{2k},
\end{equation}
where $\lambda_{max}=\sup_{d_1,\ldots,d_k}\lambda_{d_1,\ldots,d_k}$. To deal with the main term, we use the identity
\begin{equation}
\frac{1}{[d_i,e_i]}=\frac{1}{d_ie_i}\sum_{u_i\mid d_i,e_i}\varphi(u_i)
\end{equation}
and rewrite it as
\begin{equation}
\frac{X}{V}\sum_{u_1,\ldots,u_k}\left(\prod_{i=1}^k\varphi(u_i)\right)\sideset{}{^\prime}\sum_{\substack{d_1,\ldots,d_k\\e_1,\ldots,e_k\\u_i\mid d_i,e_i\forall i}}\frac{\lambda_{d_1,\ldots,d_k}\lambda_{e_1,\ldots,e_k}}{\prod_{i=1}^kd_ie_i}.
\end{equation}
By the support of $\lambda$, we may drop the requirement that $W^\prime$ is coprime to $[d_i,e_i]$. Also by the support of $\lambda$, terms with $(d_i,d_j)>1$ with $i\neq j$ have no contribution. Thus our restrictions boil down to $(d_i,e_j)=1$ for $i\neq j$. We may remove this requirement by multiplying our expression with $\sum_{s_{i,j}\mid d_i,e_j}\mu(s_{i,j})$ for all $i,j$. Then our main term becomes
\begin{equation}
\frac{X}{V}\sum_{u_1,\ldots,u_k}\left(\prod_{i=1}^k\varphi(u_i)\right)\sum_{s_{1,2},\ldots,s_{k-1,k}}\biggl(\prod_{\substack{1\leq i,j\leq k\\i\neq j}}\mu(s_{i,j})\biggr)\sum_{\substack{d_1,\ldots,d_k\\e_1,\ldots,e_k\\u_i\mid d_i,e_i\forall i\\s_{i,j}\mid d_i,e_j\forall i\neq j}}\frac{\lambda_{d_1,\ldots,d_k}\lambda_{e_1,\ldots,e_k}}{\prod_{i=1}^kd_ie_i}.
\end{equation}
We may restrict $s_{i,j}$ to be coprime to $u_i$, $u_j$, $s_{i,a}$ and $s_{b,j}$ for all $a\neq i$ and $b\neq j$ since these have no contribution by the support of $\lambda$. We denote the summation with these restrictions by $\sum^*$. We introduce the change of variable
\begin{equation}
y_{r_1,\ldots,r_k}=\biggl(\prod_{i=1}^k\mu(r_i)\varphi(r_i)\biggr)\sum_{\substack{d_1,\ldots,d_k\\r_i\mid d_i\forall i}}\frac{\lambda_{d_1,\ldots,d_k}}{\prod_{i=1}^kd_i}.
\end{equation}
Thus $y_{r_1,\ldots,r_k}$ is supported on $r=\prod_ir_i<R$, $(r,VP_f)=1$ and $\mu(r)^2=1$. This change is invertible and we have
\begin{equation}\label{iden}
\sum_{\substack{r_1,\ldots,r_k\\d_i\mid r_i\forall i}}\frac{y_{r_1,\ldots,r_k}}{\prod_{i=1}^k\varphi(r_i)}=\frac{\lambda_{d_1,\ldots,d_k}}{\prod_{i=1}^k\mu(d_i)d_i}.
\end{equation}
Hence any choice of $y_{r_1,\ldots,r_k}$ with the above mentioned support will yield a choice of $\lambda_{d_1,\ldots,d_k}$. We note here that Maynard's estimate of $\lambda_{max}$ in terms of $y_{max}=\sup_{r_1,\ldots,r_k}y_{r_1,\ldots,r_k}$ holds verbatim and we have
\begin{equation}
\lambda_{max}\ll y_{max}(\log X)^k.
\end{equation}
So our error term (\ref{err}) is $O(y_{max}^2R^2(\log X)^{4k})$. Using our change of variables we obtain
\begin{equation}
\begin{split}
S_1&=\frac{X}{V}\sum_{u_1,\ldots,u_k}\left(\prod_{i=1}^k\varphi(u_i)\right)\sideset{}{^*}\sum_{s_{1,2},\ldots,s_{k-1,k}}\biggl(\prod_{\substack{1\leq i,j\leq k\\i\neq j}}\mu(s_{i,j})\biggr)\\
&\times\biggl(\prod_{i=1}^k\frac{\mu(a_i)\mu(b_i)}{\varphi(a_i)\varphi(b_i)}\biggr)y_{a_1,\ldots,a_k}y_{b_1,\ldots,b_k}+O\left(y_{max}^2R^2(\log X)^{4k}\right),
\end{split}
\end{equation}
where $a_j=u_j\prod_{i\neq j}s_{j,i}$ and $b_j=u_j\prod_{i\neq j}s_{i,j}$. Since there's no contribution when $a_j$ or $b_j$ are not squarefree, we may rewrite $\mu(a_j)$ as $\mu(u_j)\prod_{i\neq j}\mu(s_{j,i})$, and similarly for $\varphi(a_j)$, $\mu(b_j)$ and $\varphi(b_j)$. This gives us
\begin{equation}
\begin{split}
S_1&=\frac{X}{V}\sum_{u_1,\ldots,u_k}\left(\prod_{i=1}^k\frac{\mu(u_i)^2}{\varphi(u_i)}\right)\sideset{}{^*}\sum_{s_{1,2},\ldots,s_{k,k-1}}\biggl(\prod_{\substack{1\leq i,j\leq k\\i\neq j}}\frac{\mu(s_{i,j})}{\varphi(s_{i,j})^2}\biggr)y_{a_1,\ldots,a_k}y_{b_1,\ldots,b_k}\\
&+O\left(y_{max}^2R^2(\log X)^{4k}\right).
\end{split}
\end{equation}
There is no contribution from $s_{i,j}$ with $1<s_{i,j}<D_0$ because of the restricted support of $y$. The contribution when $s_{i,j}>D_0$ is
\begin{equation}
\begin{split}
&\ll\frac{y_{max}^2X}{V}\biggl(\sum_{\substack{u<R\\(u,VP_f)=1}}\frac{\mu(u_i)^2}{\varphi(u_i)}\biggr)^k\biggl(\sum_{s_{i,j}>D_0}\frac{\mu(s_{i,j})^2}{\varphi(s_{i,j})^2}\biggr)\biggl(\sum_{s>1}\frac{\mu(s)^2}{\varphi(s)^2}\biggr)^{k^2-k-1}\\
&\ll\frac{y_{max}^2\varphi(VP_f)^kX(\log X)^k}{V(VP_f)^kD_0}.
\end{split}
\end{equation}
Our previous error of $y_{max}^2R^2(\log X)^{4k}$ can be absorbed into this error, and the terms with $s_{i,j}=1$ give us our desired main term.
\end{proof}
\begin{lem}\label{S21}
Let $S_2^{(m)}$ be as defined in \textup{(\ref{S2})}, and let
\begin{equation}
y^{(m)}_{r_1,\ldots,r_k}=\Bigl(\prod_{i=1}^k\mu(r_i)g(r_i)\Bigr)\sum_{\substack{d_1,\ldots,d_k\\r_i\mid d_i\forall i\\d_m=1}}\frac{\lambda_{d_1,\ldots,d_k}}{\prod_i\varphi(d_i)},
\end{equation}
where $g$ is the totally multiplicative function defined on primes by $g(p)=p-2$. Let $y^{(m)}_{max}=\sup_{r_1,\ldots,r_k}\lvert y^{(m)}_{r_1,\ldots,r_k}\rvert$. Then for any fixed $A>0$, we have
\begin{equation}
\begin{split}
S^{(m)}_2=\frac{X}{\varphi(V)\log X}\sum_{u_1,\ldots,u_k}\frac{(y^{(m)}_{u_1,\ldots,u_k})^2}{\prod_{i=1}^kg(u_i)}&+O\left(\frac{(y^{(m)}_{max})^2\varphi(VP_f)^{k-1}X(\log X)^{k-2}}{\varphi(V)(VP_f)^{k-1}D_0}\right)\\
&+O\left(\frac{y_{max}^2X}{\varphi(M)(\log X)^A}\right).
\end{split}
\end{equation}
\end{lem}
\begin{proof}
We first rearrange the sum to obtain
\begin{equation}
S^{(m)}_2=\sum_{\substack{d_1,\ldots,d_k\\e_1,\ldots,e_k}}\lambda_{d_1,\ldots,d_k}\lambda_{e_1,\ldots,e_k}\sum_{\substack{x\leq n<2x\\n\equiv \nu_0\pmod{W^\prime}\\ [d_i,e_i]\mid nM+h_i}}\chi_\mathbb{P}(nM+h_m).
\end{equation}
In the inner sum, if $W^\prime,[d_1,e_1],\ldots,[d_k,e_k]$ are pairwise relatively prime, the conditions determine $n$ modulo $q=W^\prime\prod_i[d_i,e_i]$, since $(M,[d_i,e_i])=1$ by the support of $\lambda$. In turn, $nM+h_m$ is determined modulo $qM=V\prod_i[d_i,e_i]$. Note that here $(q,P_f)=1$. Also, if $([d_i,e_i],nM+h_m)>1$ with $i\neq m$, then $p\mid\lvert h_i-h_m\rvert=fM$ for some $p\mid[d_i,e_i]$ and $f<D_0$ by the diameter of $\mathcal{H}$, and since $d_i$ and $e_i$ are relatively prime to both $M$ and $W$ by the support of $\lambda$, this is not possible. So $nM+h_m$ is relatively prime to the modulus if and only if $d_m=e_m=1$. Thus we can write
\begin{equation}\label{S2eq1}
\begin{split}
\sum_{\substack{x\leq n<2x\\n\equiv \nu_0\pmod{W^\prime}\\ [d_i,e_i]\mid nM+h_i}}\chi_\mathbb{P}(nM+h_m)&=\sum_{\substack{X+h_m\leq n<2X+h_m\\n\equiv b\pmod{qM}}}\chi_\mathbb{P}(n)\\&=\frac{\mathcal{P}_X}{\varphi(V)\prod_i\varphi([d_i,e_i])}+E(X,qM)+O(1),
\end{split}
\end{equation}
where
\begin{equation}
E(X,qM)=\Big\lvert\sum_{\substack{X\leq n<2X\\n\equiv b\pmod{qM}}}\chi_\mathbb{P}(n)-\frac{\mathcal{P}_X}{\varphi(qM)}\Big\rvert,
\end{equation}
$\mathcal{P}_X$ is the number of primes in $[X,2X]$, and the $O(1)$ term arises from ignoring the shift by $h_m$ in the sum. Thus the main term becomes
\begin{equation}
\frac{\mathcal{P}_X}{\varphi(V)}\sideset{}{^\prime}\sum_{\substack{d_1,\ldots,d_k\\e_1,\ldots,e_k}}\frac{\lambda_{d_1,\ldots,d_k}\lambda_{e_1,\ldots,e_k}}{\prod_i\varphi([d_i,e_i])}
\end{equation}
where $\sum^\prime$ denotes the constraint that $W^\prime,[d_1,e_1],\ldots,[d_k,e_k]$ are pairwise relatively prime. As before, there's no contribution when $(W^\prime,[d_i,e_i])>1$ or $(d_i,d_j)>1$, and we remove the conditions $(d_i,e_j)=1$ by multiplying our expression by $\sum_{s_{i,j}\mid d_i,e_j}\mu(s_{i,j})$. We also use the identity (valid for squarefree $d_i$ and $e_i$),
\begin{equation}
\frac{1}{\varphi([d_i,e_i])}=\frac{1}{\varphi(d_i)\varphi(e_i)}\sum_{u_i\mid d_i,e_i}g(u_i),
\end{equation}
where $g$ is the totally multiplicative function defined on primes by $g(p)=p-2$. The main term then becomes
\begin{equation}
\frac{\mathcal{P}_X}{\varphi(V)}\sum_{u_1,\ldots,u_k}\Bigl(\prod_{i=1}^kg(u_i)\Bigr)\sum_{s_{1,2},\ldots,s_{k-1,k}}\Bigl(\prod_{\substack{1\leq i,j\leq k\\i\neq j}}\mu(s_{i,j})\Bigr)\sum_{\substack{d_1,\ldots,d_k\\e_1,\ldots,e_k\\u_i\mid d_i,e_i\forall i\\s_{i,j}\mid d_i,e_j \forall i\neq j\\d_m=e_m=1}}\frac{\lambda_{d_1,\ldots,d_k}\lambda_{e_1,\ldots,e_k}}{\prod_i\varphi(d_i)\varphi(e_i)}.
\end{equation}
We again restrict $s_{i,j}$ to be coprime to $u_i$, $u_j$, $s_{i,a}$ and $s_{b,j}$ for all $a\neq i$ and $b\neq j$ as before, and make the change of variable
\begin{equation}\label{iden2}
y^{(m)}_{r_1,\ldots,r_k}=\Bigl(\prod_{i=1}^k\mu(r_i)g(r_i)\Bigr)\sum_{\substack{d_1,\ldots,d_k\\r_i\mid d_i\forall i\\d_m=1}}\frac{\lambda_{d_1,\ldots,d_k}}{\prod_i\varphi(d_i)}.
\end{equation}
This is invertible, and $y^{(m)}_{r_1,\ldots,r_k}$ is supported on $(\prod_ir_i,VP_f)=1$, $\prod_ir_i<R$, $\mu(\prod_ir_i)^2=1$ and $r_m=1$. Then the main term becomes
\begin{equation}
\frac{\mathcal{P}_X}{\varphi(V)}\sum_{u_1,\ldots,u_k}\Bigl(\prod_{i=1}^k\frac{\mu(u_i)^2}{g(u_i)}\Bigr)\sideset{}{^*}\sum_{s_{1,2},\ldots,s_{k-1,k}}\Bigl(\prod_{\substack{1\leq i,j\leq k\\i\neq j}}\frac{\mu(s_{i,j})}{g(s_{i,j})^2}\Bigr)y^{(m)}_{a_1,\ldots,a_k}y^{(m)}_{b_1,\ldots,b_k},
\end{equation}
where $a_j=u_j\prod_{i\neq j}s_{j,i}$ and $b_j=u_j\prod_{i\neq j}s_{i,j}$ for each $1\leq j\leq k$. Because of the restricted support of $y$, there is no contribution from terms with $(s_{i,j},VP_f)>1$. So we only need to consider $s_{i,j}=1$ or $s_{i,j}>D_0$. The contribution when $s_{i,j}>D_0$ is
\begin{align}
&\ll\frac{(y^{(m)}_{max})^2X}{\varphi(V)\log X}\Bigl(\sum_{\substack{u<R\\(u,VP_f)=1}}\frac{\mu(u)^2}{g(u)}\Bigr)^{k-1}\Bigl(\sum_s\frac{\mu(s)^2}{g(s)^2}\Bigr)^{k(k-1)-1}\sum_{s_{i,j}>D_0}\frac{\mu(s_{i,j})^2}{g(s_{i,j})^2}\\
&\ll\frac{(y^{(m)}_{max})^2\varphi(VP_f)^{k-1}X(\log X)^{k-2}}{\varphi(V)(VP_f)^{k-1}D_0}.\label{errS21}
\end{align}
The contribution from $s_{i,j}=1$ gives us the main term which is
\begin{equation}
\frac{\mathcal{P}_X}{\varphi(V)}\sum_{u_1,\ldots,u_k}\frac{(y^{(m)}_{u_1,\ldots,u_k})^2}{\prod_{i=1}^kg(u_i)}.
\end{equation}
By the prime number theorem, $\mathcal{P}_X=X/\log X+O(X/(\log X)^2))$, and the error here contributes
\begin{equation}
\frac{(y^{(m)}_{max})^2X}{\varphi(V)(\log X)^2}\Bigl(\sum_{\substack{u<R\\(u,VP_f)=1}}\frac{\mu(u)^2}{\varphi(u)}\Bigl)^{k-1}\ll\frac{(y^{(m)}_{max})^2\varphi(VP_f)^{k-1}X(\log X)^{k-3}}{\varphi(V)(VP_f)^{k-1}},
\end{equation}
which can be absorbed in the error term from (\ref{errS21}).

Now we turn to the contribution of the error terms in (\ref{S2eq1}), which is
\begin{equation}
\ll\sum_{\substack{d_1,\ldots,d_k\\e_1,\ldots,e_k}}\lvert\lambda_{d_1,\ldots,d_k}\lambda_{e_1,\ldots,e_k}\rvert\left(E(X,qM)+1\right).
\end{equation}
From the support of $\lambda$, we see that we only need to consider square-free $q$ with $q<W^\prime R^2$ and $(q,MP_f)=1$. Since for a square-free integer $q$ there are at most $\tau_{3k}(q)$ choices of $d_1,\ldots,d_k,e_1,\ldots,e_k$ for which $q=W^\prime\prod_i[d_i,e_i]$, we see that the error is
\begin{equation}
\ll\lambda^2_{max}\sum_{\substack{q<W^\prime R^2\\(q,MP_f)=1}}\mu(q)^2\tau_{3k}(q)E(X,qM)+\lambda^2_{max}\sum_{\substack{q<W^\prime R^2\\(q,MP_f)=1}}\mu(q)^2\tau_{3k}(q).
\end{equation}
Now the second term is $\ll \lambda_{max}^2W^\prime R^2\log(W^\prime R^2)^{3k-1}$. For the first term we use the Cauchy-Schwarz inequality and the trivial bound $E(X,qM)\ll X/\varphi(qM)$ to see that it is
\begin{equation}
\ll\frac{\lambda^2_{max}}{\varphi(M)^{1/2}}\biggl(\sum_{\substack{q<W^\prime R^2\\(q,MP_f)=1}}\mu(q)^2\tau_{3k}^2(q)\frac{X}{\varphi(q)}\biggr)^{1/2}\biggl(\sum_{\substack{q<W^\prime R^2\\(q,MP_f)=1}}\mu(q)^2E(X,qM)\biggr)^{1/2}.
\end{equation}
The first sum is $\ll X\log(W^\prime R^2)^{3k}$. Now for $X$ large enough, $W^\prime R^2\leq X^{e_f-\delta}$, so that Theorem \ref{UnifiedBV} applies to yield that the second sum is $\ll\frac{X}{\varphi(M)}(\log X)^{-A}$ for $A$ arbitrarily large. Thus the total contribution is
\begin{equation}
\ll \frac{y_{max}^2X}{\varphi(M)(\log X)^{A}},
\end{equation}
and this completes the proof.
\end{proof}
\begin{lem}\label{ymy}
If $r_m=1$,
\begin{equation}
y^{(m)}_{r_1,\ldots,r_k}=\sum_{a_m}\frac{y_{r_1,\ldots,r_{m-1},a_m,r_{m+1},\ldots,r_k}}{\varphi(a_m)}+O\left(\frac{y_{max}\varphi(VP_f)\log X}{VP_fD_0}\right).
\end{equation}
\end{lem}
\begin{proof}
We assume that $r_m=1$. We substitute (\ref{iden}) into (\ref{iden2}) and obtain
\begin{equation}
y^{(m)}_{r_1,\ldots,r_k}=\Bigl(\prod_{i=1}^k\mu(r_i)g(r_i)\Bigr)\sum_{\substack{d_1,\ldots,d_k\\r_i\mid d_i\forall i\\d_m=1}}\biggl(\prod_{i=1}^k\frac{\mu(d_i)d_i}{\varphi(d_i)}\biggr)\sum_{\substack{a_1,\ldots,a_k\\d_i\mid a_i\forall i}}\frac{y_{a_1,\ldots,a_k}}{\prod_{i=1}^k\varphi(a_i)}.
\end{equation}
Swapping summations over $d$ and $a$, we have
\begin{equation}
y^{(m)}_{r_1,\ldots,r_k}=\Bigl(\prod_{i=1}^k\mu(r_i)g(r_i)\Bigr)\sum_{\substack{a_1,\ldots,a_k\\r_i\mid a_i\forall i}}\frac{y_{a_1,\ldots,a_k}}{\prod_{i=1}^k\varphi(a_i)}\sum_{\substack{d_1,\ldots,d_k\\d_i\mid a_i,r_i\mid d_i\forall i\\d_m=1}}\prod_{i=1}^k\frac{\mu(d_i)d_i}{\varphi(d_i)}.
\end{equation}
The inner sum can be directly computed when $a_i$ is squarefree, which is the only case that matters by the support of $y$. We have
\begin{align}
\sum_{d_i\mid a_i,r_i\mid d_i}\frac{\mu(d_i)d_i}{\varphi(d_i)}&=\frac{\mu(r_i)r_i}{\varphi(r_i)}\sum_{d_i\mid\frac{a_i}{r_i}}\frac{\mu(d_i)d_i}{\varphi(d_i)}\\
&=\frac{\mu(r_i)r_i}{\varphi(r_i)}\prod_{p\mid\frac{a_i}{r_i}}\frac{-1}{p-1}\\
&=\frac{\mu(r_i)r_i}{\varphi(r_i)}\frac{\mu(a_i/r_i)}{\varphi(a_i/r_i)}=\frac{\mu(a_i)r_i}{\varphi(a_i)}.
\end{align}
Hence
\begin{equation}
y^{(m)}_{r_1,\ldots,r_k}=\Bigl(\prod_{i=1}^k\mu(r_i)g(r_i)\Bigr)\sum_{\substack{a_1,\ldots,a_k\\r_i\mid a_i\forall i}}\frac{y_{a_1,\ldots,a_k}}{\prod_{i=1}^k\varphi(a_i)}\prod_{i\neq m}\frac{\mu(a_i)r_i}{\varphi(a_i)}.
\end{equation}
By the support of $y$, we need only consider $a_j$ with $(a_j,VP_f)=1$. This implies $a_j=r_j$ or $a_j>D_0r_j$. The total contribution from $a_j\neq r_j$ when $j\neq m$ is
\begin{equation}
\begin{split}
&\ll y_{max}\left(\prod_{i=1}^kg(r_i)r_i\right)\!\Biggl(\sum_{a_j>D_0r_j}\frac{\mu(a_j)^2}{\varphi(a_j)^2}\Biggr)\!\Biggl(\sum_{\substack{a_m<R\\(a_m,VP_f)=1}}\frac{\mu(a_j)^2}{\varphi(a_j)}\Biggr)\prod_{\substack{1\leq i\leq k\\i\neq j,m}}\Biggl(\sum_{r_i\mid a_i}\frac{\mu(a_i)^2}{\varphi(a_i)^2}\Biggr)\\
&\ll \left(\prod_{i=1}^k\frac{g(r_i)r_i}{\varphi(r_i)^2}\right)\frac{y_{max}\varphi(VP_f)\log R}{VP_fD_0}\ll\frac{y_{max}\varphi(VP_f)\log X}{VP_fD_0}.
\end{split}
\end{equation}
Thus we find that
\begin{equation}
y^{(m)}_{r_1,\ldots,r_k}=\left(\prod_{i=1}^k\frac{g(r_i)r_i}{\varphi(r_i)^2}\right)\sum_{a_m}\frac{y_{r_1,\ldots,r_{m-1},a_m,r_{m+1},\ldots,r_k}}{\varphi(a_m)}+O\left(\frac{y_{max}\varphi(VP_f)\log X}{VP_fD_0}\right).
\end{equation}
Since the product is $1+O(D_0^{-1})$, we have the result.
\end{proof}
\begin{lem}\label{S12}
Let $y_{r_1,\ldots,r_k}$ be given in terms of a piecewise differentiable function $F$ supported on $\mathcal{R}_k=\{(x_1,\ldots,x_k)\in[0,1]^k:\sum_{i=1}^kx_i=1\}$ by
\begin{equation}\label{yfdef}
y_{r_1,\ldots,r_k}=F\bigl(\frac{\log r_1}{\log R},\ldots,\frac{\log r_k}{\log R}\bigr)
\end{equation}
whenever $r=\prod_ir_i$ is squarefree and satisfies $(r,VP_f)=1$.
Put
\begin{equation}
F_{max}=\sup_{(t_1,\ldots,t_k)\in[0,1]^k}\lvert F(t_1,\ldots,t_k)\rvert+\sum_{i=1}^k\lvert\frac{\partial F}{\partial t_i}(t_1,\ldots,t_k)\rvert.
\end{equation}
Then
\begin{equation}
S_1=\frac{\varphi(VP_f)^kX(\log R)^k}{V(VP_f)^k}I_k(F)+O\Bigl(\frac{F_{max}^2\varphi(VP_f)^kX(\log X)^{k-1}\log\log X}{V(VP_f)^kD_0}\Bigr),
\end{equation}
where
\begin{equation}
I_k(F)=\int_0^1\cdots\int_0^1F(t_1,\ldots,t_k)^2dt_1\ldots dt_k.
\end{equation}
\end{lem}
\begin{proof}
We substitute (\ref{yfdef}) into our expression for $S_1$ from Lemma \ref{S11} and obtain
\begin{equation}
\begin{split}
S_1=&\frac{X}{V}\sum_{\substack{u_1,\ldots,u_k\\(u_i,u_j)=1,\forall i\neq j\\(u_i,VP_f)=1\forall i}}\biggl(\prod_{i=1}^k\frac{\mu(u_i)^2}{\varphi(u_i)}\biggr)F\bigl(\frac{\log u_1}{\log R},\ldots,\frac{\log u_k}{\log R}\bigr)^2\\
&+O\left(\frac{F_{max}^2\varphi(VP_f)^kX(\log X)^k}{V(VP_f)^kD_0}\right).
\end{split}
\end{equation}
Now if $(u_i,u_j)>1$ for some $i\neq j$ and $(u_i,VP_f)=(u_j,VP_f)=1$, then there is a prime $p\mid(u_i,u_j)$ with $p\nmid VP_f$, so \emph{a fortiori} $p\nmid W$ and $p>D_0$. Thus the cost of dropping the condition $(u_i,u_j)=1$ is an error of size
\begin{equation}
\begin{split}
&\ll\frac{F_{max}^2X}{V}\sum_{p>D_0}\:\:\sum_{\substack{u_1,\ldots,u_k<R\\p\mid u_i,u_j\\(u_i,VP_f)=1\forall i}}\:\:\:\prod_{i=1}^k\frac{\mu(u_i)^2}{\varphi(u_i)}\\
&\ll\frac{F_{max}^2X}{V}\sum_{p>D_0}\frac{1}{(p-1)^2}\Bigl(\sum_{\substack{u<R\\(u,VP_f)=1}}\frac{\mu(u)^2}{\varphi(u)}\Bigr)^k\ll\frac{F_{max}^2\varphi(VP_f)^kX(\log X)^k}{V(VP_f)^kD_0}.
\end{split}
\end{equation}
Thus we are left to evaluate
\begin{equation}
\sum_{\substack{u_1,\ldots,u_k\\(u_i,VP_f)=1\forall i}}\biggl(\prod_{i=1}^k\frac{\mu(u_i)^2}{\varphi(u_i)}\biggr)F\bigl(\frac{\log u_1}{\log R},\ldots,\frac{\log u_k}{\log R}\bigr)^2.
\end{equation}
This differs from the corresponding sum in Maynard's work only in that we have a $VP_f$, which does not have as small prime factors, in place of $W$. We put,
\begin{equation}
\gamma(p)=\begin{cases}
1, \:\:&\text{if}\:\:p\nmid VP_f,\\
0, &\text{otherwise}.
\end{cases}
\end{equation}
Then we can use Lemma 6.1 of \cite{Ma2015} with $\kappa=1$,
\begin{equation}
L\ll 1+\sum_{p\mid VP_f}\frac{\log p}{p}\ll\Bigl(\sum_{p\leq \log R}+\sum_{\substack{p\mid MP_f\\p>\log R}}\Bigr)\frac{\log p}{p}\ll\log\log R +\frac{\log MP_f}{\log R}\ll\log\log X,
\end{equation}
and $A_1$ and $A_2$ suitable constants. The lemma then yields
\begin{equation}
\begin{split}
\sum_{\substack{u_1,\ldots,u_k\\(u_i,VP_f)=1\forall i}}\biggl(&\prod_{i=1}^k\frac{\mu(u_i)^2}{\varphi(u_i)}\biggr)F\bigl(\frac{\log u_1}{\log R},\ldots,\frac{\log u_k}{\log R}\bigr)^2\\
&=\frac{\varphi(VP_f)^k(\log R)^k}{(VP_f)^k}I_k(F)+O\Bigl(\frac{F_{max}^2\varphi(VP_f)^k(\log X)^{k-1}\log\log X}{(VP_f)^kD_0}\Bigr)
\end{split}
\end{equation}
with
\begin{equation}
I_k(F)=\int_0^1\cdots\int_0^1F(t_1,\ldots,t_k)^2dt_1\ldots dt_k,
\end{equation}
and the proof is complete.
\end{proof}

\begin{lem}\label{S22}
Let $y_{r_1,\ldots,r_k}$, $F$, and $F_{max}$ be as in Lemma \textup{\ref{S12}}. Then
\begin{equation}
S^{(m)}_2=\frac{\varphi(VP_f)^kX(\log R)^{k+1}}{V(VP_f)^k\log X}J^{(m)}_k(F)+O\left(\frac{F_{max}^2\varphi(VP_f)^kX(\log X)^k}{V(VP_f)^kD_0}\right),
\end{equation}
where
\begin{equation}
J_k^{(m)}(F)=\int_0^1\cdots\int_0^1\Biggl(\int_0^1F(t_1,\ldots,t_k)dt_m\Biggr)^2dt_1\ldots dt_{m-1}dt_{m+1}\ldots dt_k.
\end{equation}
\end{lem}
\begin{proof}
From Lemma \ref{S21}, we want to evaluate the sum
\begin{equation}
\sum_{u_1,\ldots,u_k}\frac{(y^{(m)}_{u_1,\ldots,u_k})^2}{\prod_{i=1}^kg(u_i)}.
\end{equation}
First we estimate $y^{(m)}_{r_1,\ldots,r_k}$. Recall that $y^{(m)}_{r_1,\ldots,r_k}$ is supported on $(\prod_ir_i,VP_f)=1$, $\mu(\prod_ir_i)^2=1$, $(r_i,r_j)=1$ when $i\neq j$ and $r_m=1$. Then substituting (\ref{yfdef}) into our expression for $y^{(m)}_{r_1,\ldots,r_k}$ from Lemma \ref{ymy}, we obtain
\begin{equation}
\begin{split}
y^{(m)}_{r_1,\ldots,r_k}=&\sum_{(u,VP_f\prod_ir_i)=1}\frac{\mu(u)^2}{\varphi(u)}F\bigl(\frac{\log r_1}{\log R},\ldots,\frac{\log r_{m-1}}{\log R},\frac{\log u}{\log R},\frac{\log r_{m+1}}{\log R},\ldots,\frac{\log r_k}{\log R}\bigr)\\
&+O\left(\frac{F_{max}\varphi(VP_f)\log X}{VP_fD_0}\right).
\end{split}
\end{equation}
From this it is plain that
\begin{equation}
y^{(m)}_{max}\ll\frac{\varphi(VP_f)}{VP_f}F_{max}\log X.
\end{equation}
Now we use Lemma 6.1 of \cite{Ma2015} again, with $\kappa=1$,
\begin{equation}
\gamma(p)=\begin{cases}
1, \:\:&\text{if}\:\:p\nmid VP_f\prod_{i=1}^kr_i,\\
0, &\text{otherwise}.
\end{cases}
\end{equation}
\begin{equation}
L\ll 1+\sum_{p\mid V\prod_ir_i}\frac{\log p}{p}\ll\Bigl(\sum_{p\leq \log R}+\sum_{\substack{p\mid MP_f\prod_ir_i\\p>\log R}}\Bigr)\frac{\log p}{p}\ll\log\log X,
\end{equation}
and $A_1$, $A_2$ suitable constants to obtain
\begin{equation}\label{ymff}
y^{(m)}_{r_1,\ldots,r_k}=(\log R)\frac{\varphi(VP_f)}{VP_f}\Bigl(\prod_{i=1}^k\frac{\varphi(r_i)}{r_i}\Bigr)F_{r_1,\ldots,r_k}^{(m)}+O\left(\frac{F_{max}\varphi(VP_f)\log X}{VP_fD_0}\right),
\end{equation}
where
\begin{equation}
F_{r_1,\ldots,r_k}^{(m)}=\int_0^1F\bigl(\frac{\log r_1}{\log R},\ldots,\frac{\log r_{m-1}}{\log R},t_m,\frac{\log r_{m+1}}{\log R},\ldots,\frac{\log r_k}{\log R}\bigr)dt_m.
\end{equation}
This is valid if $r_m=1$, and $r=\prod_{i=1}^kr_i$ satisfies $(r,VP_f)=1$ and $\mu(r)^2=1$, otherwise $y^{(m)}_{r_1,\ldots,r_k}=0$. Squared, (\ref{ymff}) gives
\begin{equation}
\begin{split}
(y^{(m)}_{r_1,\ldots,r_k})^2=&(\log R)^2\frac{\varphi(VP_f)^2}{(VP_f)^2}\Bigl(\prod_{i=1}^k\frac{\varphi(r_i)^2}{r_i^2}\Bigr)(F_{r_1,\ldots,r_k}^{(m)})^2\\
&+O\left(\frac{(F_{max})^2\varphi(VP_f)^2(\log X)^2}{(VP_f)^2D_0}\right).
\end{split}
\end{equation}
Using this in the expression for $S_2^{(m)}$ from Lemma \ref{S21}, we have
\begin{multline}
S^{(m)}_2=\frac{\varphi(VP_f)^2X(\log R)^2}{\varphi(V)(VP_f)^2\log X}\sum_{\substack{r_1,\ldots,r_k\\(r_i,VP_f)=1\\(r_i,r_j)=1\forall i\neq j\\r_m=1}}\Bigl(\prod_{i=1}^k\frac{\mu(r_i)^2\varphi(r_i)^2}{g(r_i)r_i^2}\Bigr)(F^{(m)}_{r_1,\ldots,r_k})^2\\+O\left(\frac{F_{max}^2\varphi(VP_f)^kX(\log X)^k}{V(VP_f)^kD_0}\right).
\end{multline}
We drop the condition $(r_i,r_j)=1$ as before, this time introducing an error of size
\begin{equation}
\begin{split}
&\ll\frac{F_{max}^2\varphi(VP_f)^2X(\log R)^2}{\varphi(V)(VP_f)^2\log X}\biggl(\sum_{p>D_0}\frac{\varphi(p)^4}{g(p)^2p^4}\biggr)\biggl(\sum_{\substack{r<R\\(r,VP_f)=1}}\frac{\varphi(r)^2}{g(r)r^2}\biggr)^{k-1}\\
&\ll\frac{F_{max}^2\varphi(VP_f)^{k+1}X(\log X)^k}{\varphi(V)(VP_f)^{k+1}D_0}.
\end{split}
\end{equation}
Thus we are left to evaluate
\begin{equation}
\sum_{\substack{r_1,\ldots,r_{m-1},r_{m+1},\ldots,r_k\\(r_i,VP_f)=1}}\Bigl(\prod_{i=1}^k\frac{\mu(r_i)^2\varphi(r_i)^2}{g(r_i)r_i^2}\Bigr)(F^{(m)}_{r_1,\ldots,r_k})^2.
\end{equation}
Again we apply Lemma 6.1 from Maynard with $\kappa=1$, with
\begin{equation}
\gamma(p)=\begin{cases}
1-\frac{p^2-3p+1}{p^3-p^2-2p+1}, \:\:&\text{if}\:\:p\nmid VP_f,\\
0, &\text{otherwise},
\end{cases}
\end{equation}
\begin{equation}
L\ll 1+\sum_{p\mid VP_f}\frac{\log p}{p}\ll\log\log X,
\end{equation}
and $A_1$, $A_2$ suitable constants. The singular series in this case is
\begin{equation}
\mathfrak{S}=\frac{\varphi(VP_f)}{VP_f}\bigl(1+O\bigl(\frac{1}{D_0}\bigr)\bigr),
\end{equation}
and we obtain
\begin{equation}\label{S2eq2}
S^{(m)}_2=\frac{\varphi(VP_f)^{k+1}X(\log R)^{k+1}}{\varphi(V)(VP_f)^{k+1}\log X}J^{(m)}_k(F)+O\left(\frac{F_{max}^2\varphi(VP_f)^{k+1}X(\log X)^k}{\varphi(V)(VP_f)^{k+1}D_0}\right).
\end{equation}
Now in the main term we have
\begin{equation}
\begin{split}
\frac{\varphi(VP_f)}{\varphi(V)(VP_f)}&=\frac{1}{V}\cdot\frac{V}{\varphi(V)}\cdot\frac{\varphi(VP_f)}{(VP_f)}\\
&=\frac{1}{V}\prod_{p\mid V}\frac{p}{p-1}\prod_{p\mid VP_f}\frac{p-1}{p}\\
&=\frac{1}{V}\prod_{\substack{p\mid P_f\\p\nmid V}}\frac{p-1}{p}.
\end{split}
\end{equation}
This last product is either vacuous, or consists of a single factor $(1-p_0^{-1})$, which is $1+O\left((\log\log X)^{-1}\right)$. Thus we may replace \eqref{S2eq2}, within acceptable error, with
\begin{equation}
S^{(m)}_2=\frac{\varphi(VP_f)^kX(\log R)^{k+1}}{V(VP_f)^k\log X}J^{(m)}_k(F)+O\left(\frac{F_{max}^2\varphi(VP_f)^kX(\log X)^k}{V(VP_f)^kD_0}\right),
\end{equation}
where we have replaced $\frac{\varphi(VP_f)}{\varphi(V)(VP_f)}$ with $1/V$ in the error term as well.
\end{proof}

\section{Discussion}
Baker and Zhao also consider primes in arithmetic progressions, except they prove their result for certain smooth moduli (recall that a number is called $y$-smooth if it has no prime factor exceeding $y$). The techniques they employ involve estimating Dirichlet polynomials and appealing to a zero-free region described in terms of the largest prime and the squarefree kernel of $M$ to obtain the required Bombieri-Vinogradov type theorem. Their result~\cite{BZ2014}*{Theorem 1} reads as follows (with the notation adapted where applicable to avoid confusion).
\begin{BZ*}
Let $\eta>0$, $r\geq1$, and let $M=X^\theta$ with $0<\theta\leq5/12-\eta$, $(a,M)=1$. Let
\begin{equation*}
K(\theta)=
\begin{cases}
\frac{4}{1-2\theta}\quad &\text{if $\theta<2/5-\varepsilon$,}\\
\frac{40}{9-20\theta}\quad &\text{if $\theta\geq2/5-\varepsilon$,}.
\end{cases}
\end{equation*}
Suppose that $M$ satisfies
\begin{equation*}
\max\{p:p\mid M\}<\exp\left(\frac{\log X}{B\log\log X}\right),\quad\prod_{p\mid M}p< X^\delta,\quad w\nmid M
\end{equation*}
with
\begin{equation*}
B=\frac{C_1}{\eta}\exp\left(\frac{4(r+1)}{K(\theta)}\right),\quad\delta=\frac{C_3\eta}{r+\log(1/\eta)}\exp\left(-\frac{4(r+1)}{K(\theta)}\right)
\end{equation*}
for suitable absolute positive constants $C_1$ and $C_3$, and $w$ denotes the possibly existing unique exceptional modulus to which there's a Dirichlet $L$-function with a zero in the region $\beta>c_1/\log X$. There are primes $p_n<\ldots<p_{n+r}$ in $(X/2,X]$, with $p_i\equiv a\pmod{M}$ such that
\begin{equation*}
p_{n+r}-p_n<C_2Mr\exp\left(K(\theta)r\right).
\end{equation*}
Here $C_2$ is a positive absolute constant.
\end{BZ*}
Recalling our Theorem \ref{gnrlthm},
\begin{equation*}
p_{n+r}-p_n\ll\left(\frac{r}{\eta}\right)^3\exp\left(\frac{5r}{3\eta}\right)M,
\end{equation*}
one immediately sees that the Baker-Zhao bound is stronger when $r$ grows, and also has the advantage of describing the moduli for which it holds (apart from the possibility of being a multiple of the exceptional modulus if it exists). On the other hand, as per the second remark following Proposition \ref{exception}, the result of the present work holds for $X^{5/12-\eta}\left(1-c/\log\log X\right)$ moduli up to $X^{5/12-\eta}$, while by Dickman's Theorem (see, for instance, \cite{MV2007}*{Theorem 7.2}), there are $o(X^{5/12-\eta})$ integers with no prime divisors exceeding $\exp\left(\frac{\log X}{B\log\log X}\right)$ for which the Baker-Zhao result holds. Hence the present result is valid for a much larger class of arithmetic progressions. With these considerations the two can be regarded as complementary results concerning uniform small gaps between primes in arithmetic progressions over a range of moduli.
\vspace{1em}

\noindent {\bf Acknowledgements. } The author would like to thank the American Institute of Mathematics for the opportunity to participate in the 2014 November workshop on bounded gaps between primes, from which the present work received much stimulus.

\end{document}